%% file: BHE_vs_PCflow_stability.tex
\documentclass[a4paper,11pt]{amsart}

\input{configuration}

\title{Bismut Hermitian Einstein metrics and the stability of the pluriclosed flow}

\author{Giuseppe Barbaro}
\address[Giuseppe Barbaro]{Università La Sapienza, Roma}
\email{g.barbaro@uniroma1.it}

\address{Dipartimento di Matematica ``Guido Castelnuovo", Università la Sapienza, Piazzale Aldo Moro, 5, 00185 Roma, Italy} 
\keywords{}
\thanks{The author is supported by GNSAGA of INdAM and by project PRIN2017 ``Real and Complex Manifolds: Topology, Geometry and holomorphic dynamics'' (code 2017JZ2SW5)}

\begin{document}
	
\begin{abstract}
    We compute the $(1,1)$-Aeppli cohomology of compact Bismut flat manifolds with finite fundamental group. We thus deduce that the Bismut flat metrics on these manifolds are globally stable for the pluriclosed flow. In turn, this also prevents the existence of non-flat homogeneous Bismut Hermitian Einstein (hence pluriclosed Calabi--Yau with torsion) metrics on a large class of homogeneous manifolds. 
\end{abstract}
	
\maketitle
	
\section{Introduction}
    {\em Hermitian geometry} lies at the intersection of {\em complex} and {\em Riemannian} geometry.
    Indeed, a Hermitian manifold $(M,J,g)$ is endowed with a complex structure $J$ and a compatible metric $g$, i.\,e. such that $g$ is $J$-invariant and $\omega:=g(J\cdot,\cdot)$ is a non-degenerate $(1,1)$-form.
    When the fundamental 2-form $\omega$ is closed (hence symplectic) the manifold is called K\"ahler.
    There are several reasons to study Hermitian non-K\"ahler geometry; we distinguish two of them which are central to our dissertation:
    to investigate non-K\"ahler Calabi--Yau geometries arising from theoretical physics, and completing the classification of complex surfaces.

    Several natural connections have been introduced to deal with the non-K{\"a}hler case.
    As a matter of fact, on a Hermitian manifold $(M,J,g)$ the Levi--Civita connection $\Na^{LC}$ preserves the complex structure, i.\,e. $\Na^{LC}J=0$, if and only if the metric is K{\"a}hler.
    Henceforth, one considers other connections, possibly with torsion, that preserve the Hermitian structure.
    Precisely, these connections $\Na$ satisfy $\Na J=\Na g=0$ and are called {\em Hermitian} connections.
    For what concerns this article, we focus on the {\em Bismut connection} $\Na^B$ (also known as {\em Strominger} or {\em Strominger--Bismut} connection), since it appears in the context of classification of complex surfaces through the {\em pluriclosed flow} \cite{MR2673720, MR3110582, MR4181011}, and it is used in describing models of heterotic string theory \cite{MR851702, MR800347, GATES1984157, MR872720} like {\em Calabi--Yau with torsion} geometries.

    The Bismut connection first appeared in a work of Bismut \cite{MR1006380} in the context of {\em non-K\"ahler index theory}, and it is defined on a Hermitian manifold $(M,J,g)$ with respect to the Levi--Civita connection as
    \begin{equation*}
        g\left(\nabla^{B}_{x}y,z\right)=g\left(\nabla^{LC}_{x}y,z\right)+\frac{1}{2}d\omega\left(Jx,Jy,Jz\right),
    \end{equation*}
    for $x,y,z$ vector fields on $M$.
    The crucial peculiarity of the Bismut connection is that it has a skew-symmetric torsion chosen to adapt to the Hermitian context.
    It is indeed the only Hermitian connection with skew-symmetric torsion, which is given by 
    \begin{equation*}
		g\left(T^B(\cdot,\cdot),\cdot\right) = d\omega(J\cdot,J\cdot,J\cdot) = -Jd\omega(\cdot,\cdot,\cdot).
    \end{equation*}
    When the torsion of the Bismut connection is closed, that is $dJd\omega=0$, the metric is said {\em pluriclosed} (or {\em strong K\"ahler with torsion}, SKT in short).\\    

    In~\cite{MR2673720} the authors introduced an evolution equation for Hermitian metrics that preserves the pluriclosed condition called {\em pluriclosed flow}.
    It evolves a starting pluriclosed metric $\omega_0$ as
    \begin{equation*}
        \begin{cases}
            \frac{\p }{\p t}\omega (t) = - \left(Ric^B(\omega(t))\right)^{1,1},\\
            \omega(0)= \omega_0, 
        \end{cases}
    \end{equation*}
    where $Ric^B(g)$ is given by tracing the Bismut curvature tensor $R^B(g)$ in the endomorphism components, $\left(\cdot\right)^{1,1}$ denotes the $(1,1)$-component of the form, and the flow preserves the pluriclosed condition since $\p\bar\p\left(Ric^B(\omega(t))\right)^{1,1}=0 $ by \eqref{eq: ricci bismut e chern}.
    The pluriclosed flow is particularly well-suited for compact complex surfaces since there is a pluriclosed representative in any conformal class of Hermitian metrics \cite[Théorème 1]{MR470920}.
    As a matter of fact, it has been introduced as an analytic tool to understand the topology and geometry of compact complex surfaces.
    In particular, it is conjectured~\cite[Section~5]{MR3110582} that on a minimal {\em class VII surface} with $b_2>0$ it should detect a rational curve, leading to a parabolic proof of the {\em Global Spherical Shell conjecture}~\cite{MR780359} for $b_2=1$.
    In~\cite{MR4181011} further analysis of the conjectural behavior of the pluriclosed flow on class VII surfaces has been made; see also \cite[Section 7.3]{MR4579184} and the references therein for a survey on its possible application to the classification of complex surfaces.

    Up to now, there is no definitive understanding of the behaviour of the pluriclosed flow and the singularities it may encounter.
    In fact, a complete description of the long-time evolution of this flow has been reached only in the locally homogeneous setting for complex surfaces \cite{MR3511471}, and for metrics of non-positive, flat, or negative holomorphic bisectional curvature \cite{MR3462132}.
    In \cite{MR3957836}, \cite{MR4288257} and \cite{fusi_vez} it has been developed an analysis of the behavior of the pluriclosed flow acting on invariant metrics on {almost-abelian Lie groups}, {Hopf manifolds}, and {Oeljeklaus--Toma manifolds} respectively. 
    In this article we address the long-time existence and convergence of the pluriclosed flow on compact {\em Bismut flat manifolds}, which are complex manifolds equipped with a Hermitian metric whose Bismut curvature tensor vanishes, i.\,e. $R^B(g)\equiv 0$. We prove the following result.
    \begin{theo*}[Theorem \ref{th: global stability bis}]
    	Let $(M,J,\omega_{BF})$ be a compact Bismut flat manifold with finite fundamental group.
    	Then for any pluriclosed metric $\omega_0$ on $\left(M,J\right)$ the solution to the pluriclosed flow with initial data $\omega_0$ exists on $[0, \infty)$ and converges to a Bismut flat metric $\omega_{\infty}$.
    \end{theo*}
	The proof is based on a previous result obtained in \cite{Garcia-Fernandez:2021tq} (and stated here as Theorem \ref{thm: GJS convergence}) which ensures that given a pluriclosed metric  $\omega_0$ such that 
	\begin{equation}\tag{Tor-class}\label{eq: cohomological condition}
		[\p\omega_0] = [\p\omega_{BF} ] \in H^{2,1}_{\bar\p}(M,J),
	\end{equation}
	the solution of the pluriclosed flow with initial data $\omega_0$ exists on $[0, \infty)$ and converges to a Bismut flat metric $\omega_\infty$.


    It is evident that 
    if \eqref{eq: cohomological condition} is satisfied by any pluriclosed metric, then the Bismut flat metrics are {\em globally-stable static points} of the pluriclosed flow.
    As a consequence, Theorem \ref{thm: GJS convergence} becomes powerful when combined with the knowledge of the cohomology of the manifold.
    We verify that for compact Bismut flat manifold with finite fundamental group the knowledge of the dimension of the {\em $(1,1)$-Aeppli cohomology} group is enough to derive the global stability of the Bismut flat metrics.
    As a matter of fact, the cohomological condition of Theorem \ref{thm: GJS convergence} can be restated in terms of the Aeppli cohomology, which is better suited for studying pluriclosed metrics.
    Then, we reduce the problem of global stability to prove that the $(1,1)$-Aeppli cohomology is generated by the classes of Bismut flat metrics.
    This idea was also used in some specific cases such us the Hopf surface, the Calabi--Eckmann threefold, and the compact simple Lie groups of rank two, respectively in \cite[Examples 2.7 and 2.8]{Garcia-Fernandez:2021tq} and \cite{barbaro_pcflow}.
    The novelty of this article is to exploit the torus bundle structure of Bismut flat manifolds to apply Chern--Weyl theory in computing their Aeppli cohomology.
    Remarkably, our approach avoid the computation of the Dolbeault cohomology of the manifolds, that was instead the usual argument, see Remark \ref{rmk: crucial}.
    
    We also highlight that from the argument in \cite{Garcia-Fernandez:2021tq} one obtains that the flow is converging to a Bismut flat metric, but there is no evident relation between the endpoint of the flow $\omega_\infty$ and the reference metric $\omega_{BF}$.
    However, with Theorem \ref{th: global stability}, we show that, a posteriori, we can obtain a better understanding of the relation between $\omega_\infty$ and $\omega_{BF}$.    

    For compact Bismut flat manifolds in general, the problem of global stability remains open.
    In Example \ref{ex: higher rank} we explicitly construct a compact Bismut flat manifold with infinite fundamental group to highlight the main difficulties in extending our argument.
    As far as we know, this is also the first example of a compact complex manifold $(M,J)$ equipped with a family of Bismut flat metrics $\left\{\omega_u\right\}_{u\in\C}$ which are diffeomorphic one to each other but have different torsion classes, i.\,e. $[\p\omega_u]\neq[\p\omega_{u'}]$ in $H^{2,1}_{\bar\p}(M,J)$. \\
    
    Another essential step in comprehending the evolution of the pluriclosed flow involves gaining insights into the geometry of its stationary points.
    These are pluriclosed metrics $\omega$ which satisfy
    \begin{equation}\label{1,1 Bismut Einstein}\tag{BHE}
        \left(Ric^B(\omega)\right)^{1,1} = \lambda \omega \, ,\quad\lambda\in\R.
    \end{equation}
    It turns out \cite[Theorem 1.4]{ye2023bismut}, that \eqref{1,1 Bismut Einstein} forces either the metric to be K\"ahler or the Bismut--Ricci form to vanish.
    In the latter case, these metrics are 
    are called {\em Bismut Hermitian Einstein} \cite[Definition 2.4]{Garcia-Fernandez:2021tq}.
    Up to now, the only known Bismut Hermitian Einstein metrics are the Bismut flat metrics, the K{\"a}hler Einstein metrics, and their Hermitian products, motivating the interest in finding non-K\"ahler non-flat examples.
    As a consequence of our main Theorem \ref{th: global stability bis} we show that on a large class of homogeneous manifolds there are no non-trivial Bismut Hermitian Einstein metrics.
    More precisely, we study {\em C-spaces}, which were defined by Wang \cite{MR66011} as compact complex manifolds admitting a transitive action by a compact Lie group $G$ of biholomorphisms and finite fundamental group, and we obtain the following result generalizing Corollary 5.1 in \cite{Fino22}.
    \begin{theo*}[Theorem \ref{thm: non-existence on C-spaces}]
    	Let $(M,J)$ be a C-space and suppose that $g$ is a left-invariant Bismut Hermitian Einstein metric on it.
    	Then, up to finite cover, $M$ is a Lie group and $g$ is a bi-invariant metric.
    	In particular, $g$ is Bismut flat.
    \end{theo*}
	The idea behind the proof is to use the results in \cite{MR4032184} and \cite{Fino22} to show that the manifold is forced to admit also a Bismut flat metric.
	Then the existence of non-flat Bismut Hermitian Einstein metrics is in contrast with the global stability of the Bismut flat metrics for the pluriclosed flow.

	Finally, we recall that a Hermitian metric $g$ is said to be {\em Calabi--Yau with torsion} (CYT in short) if $Ric^B(g)=0$.
	Therefore, the Bismut Hermitian Einstein equation is given by the pairing of the pluriclosed and the CYT conditions.
	However, Calabi--Yau with torsion geometry is interesting on its own since it plays a role in Physics after the works of Strominger \cite{MR851702} and Hull \cite{MR862401}, and it attracted attention as models for string compactifications, see e.g. \cite{MR1954814, MR1989858, MR1959324, MR2030780, MR2095098, MR2096734}.
	
	Being a Hermitian connection, $\Na^B$ determines a representative of the first Chern class $c_1(M)$ in de Rham cohomology via its Ricci form.
	Therefore, the existence of Calabi--Yau with torsion structures is obstructed by having vanishing first Chern class.
	It turns out \cite[Theorem 3]{MR2764884} that every compact complex homogeneous space with vanishing first Chern class, after an appropriate deformation of the complex structure, admits a homogeneous CYT metric, provided that the complex homogeneous space also has an invariant volume form.
	Then, even if the CYT metrics are generally not unique, in \cite[Theorems 4.6 and 4.7]{MR4554058} it was proved the uniquenes among the homogeneous metrics on class $\mathcal{C}$ manifolds in the sense of \cite{MR3869430} (with the exception of some specific cases).
	Therefore, the search for non-trivial Bismut Hermitian Einstein metrics on homogeneous spaces is not just the first natural attempt, but it is also motivated by these results about the existence and uniqueness of homogeneous CYT structures.
	
	\medskip
	
	The paper is organized as follows.
	In Section \ref{sec: 2}, we present the characterizzation of Bismut flat manifolds following \cite{MR4127891}; we thus recall the construction of the complex structures on them as in \cite{MR59287}.
	Section \ref{sec: cohomology comp} is devoted to the computation of the $(1,1)$-Aeppli cohomology of the universal cover of the compact Bismut flat manifolds with finite fundamental group.
	Then, in Section \ref{sec: stability}, we use it to prove the global stability of the pluriclosed flow on these manifolds.
	Finally, in Section \ref{sec: non-flat BHE}, we use this stability property to derive the non-existence of non-flat homogeneous Bismut Hermitian Einstein metrics on C-spaces.
	
\section*{Acknowledgements} 
	I would like to thank Francesco Pediconi and Jonas Stelzig for the help provided during this work and for
	useful suggestions on the topics of this paper.
	I am also grateful to Daniele Angella for his constant support and encouragement.

\section{Complex structure on Bismut flat manifolds}\label{sec: 2}
    By an explicit construction (briefly described in Section \ref{subsec: sam constr}), Samelson showed \cite{MR59287} that any even-dimensional compact Lie group admits a left-invariant complex structure compatible with the bi-invariant metric coming from the Killing form. 
    Moreover, Alexandrov and Ivanov \cite{MR1836272} proved that any even dimensional connected Lie group equipped with a bi-invariant metric $g$ and a left-invariant complex structure which is compatible with $g$ is Bismut flat. 
    Afterward, Wang, Yang, and Zheng \cite{MR4127891} showed that up to taking the universal cover, these are the only existing compact Bismut flat manifolds. 
    In other words, simply-connected compact Bismut flat manifolds have been characterized as {\em Samelson spaces}, whose definition is as follows.
    \begin{defi}[\cite{MR4127891}]\label{def: samelson space}
		A Samelson space is a Hermitian manifold $(G, g, J)$, where $G$ is a connected and simply-connected, even-dimensional Lie group, $g$ a bi-invariant metric on $G$, and $J$ a left-invariant complex structure on $G$ that is compatible with $g$.
    \end{defi}
    By Milnor's Lemma \cite[Lemma 7.5]{MR425012}, a simply-connected Lie group $G'$ with a bi-invariant metric must be the product of a compact semisimple Lie group with an additive vector group.
    \begin{lem}[Lemma 7.5 of \cite{MR425012}]\label{lem: Milnor}
		Let $G$ be a simply-connected Lie group with a bi-invariant metric $\langle \cdot,\cdot \rangle$. Then $G$ is isomorphic and isometric to the product $G_1 \times \cdots \times G_r \times {\mathbb R}^k$ where each $G_i$ is a simply-connected compact simple Lie group and ${\mathbb R}^k$ is the additive vector group with the flat metric.
    \end{lem}
    Taking quotients of these manifolds one obtains the {\em local Samelson spaces}.
    \begin{defi}[\cite{MR4127891}]\label{def: loc sam space}
		Let $(G', g, J)$ be a Samelson space, where $G' = G \times \R^k$ with $G$ semisimple. 
		Let $\rho: \Z^k \rightarrow I(G)$ be a homomorphism into the isometry group of $G$. 
		Then $\Gamma_\rho \sim \Z^k$ acts on $G \times \R^k$ by $\gamma(x, y) = (\rho(\gamma)(x), y + \gamma)$ as isometries, and it acts freely and properly discontinuously, so one gets a compact quotient $M_\rho = (G \times \R^k)/\Gamma_\rho$. 
		If the complex structure of $G'$ is preserved by $\Gamma_\rho$, then it descends down to $M_\rho$ and makes it a complex manifold. 
		In this case, the compact Hermitian manifold $M_\rho$ is called local Samelson space. 
    \end{defi}
	Such Hermitian manifolds are Bismut flat since their universal cover is so.
	Moreover, they are all the compact Bismut flat manifolds up to finite cover.
    \begin{theo}[Theorem 1 in \cite{MR4127891}]\label{thm: Bismut flat manifolds}
		Let $(M,J, g)$ be a compact Hermitian manifold whose Bismut connection is flat.
		Then there exists a finite cover $M'$ of $M$ such that $M'$ is a local Samelson space $M_\rho$ defined as above. 
		Also, $M_\rho$ is diffeomorphic to $G \times \left(\S^1\right)^k$. 
    \end{theo}
	
	
	
	\medskip
	
    In \cite{MR994129} Pittie gave a complete description of the moduli of left-invariant complex structures on even-dimensional compact Lie groups, proving that they all come from Samelson's construction in \cite{MR59287}, namely, from a choice of a maximal torus, a complex structure on the Lie algebra of the torus, and a choice of positive roots for the {\em Cartan decomposition}. 
    Let us now recall this construction in more detail. 
	
\subsection{Cartan decomposition}\label{subsec: cartan}
    Let $G$ be an even-dimensional connected Lie group and denote by ${\mathfrak g}$ the Lie algebra of $G$ and by ${\mathfrak g}^\C $ its complexification.   
    We recall that the Killing form of $G$ is given by
    \begin{equation*}\label{eq: killing}
		B(X,Y) := \tr \left(\mbox{ad}_X \circ\mbox{ad}_Y\right),
    \end{equation*}
    where $\mbox{ad}_X(Y)=[X,Y]$ for $X,Y\in\g$.
    This is a symmetric bilinear form, and thanks to the Cartan criterion \cite[Théorème 1 in Chapitre IV]{MR1505696} it is non-degenerate if and only if $\g$ is semisimple.
    If this is the case, its opposite is a metric which we indicate with $\left< \cdot,\cdot\right>:=-B(\cdot,\cdot)$.
    Moreover, if $G$ is a simple Lie group then any invariant symmetric bilinear form on it is a scalar multiple of the Killing form.
	
    Given $K$ a maximal torus of $G$, we denote by $\k$ its Lie algebra.
    The action of $\mbox{ad}\,\k$ on $\g$ is simultaneously diagonalizable \cite[Corollary 2.23]{MR1920389}, and the eigenvalues of $\mbox{ad}\,\k$ on $\g$ are called {\em roots} of $\g$ with respect to $\k$.    
    Then one has the $\mbox{ad}(K)$-invariant {\em roots decomposition}
    \begin{equation*} 
		{\mathfrak g}^\C = {\mathfrak k}^\C \oplus \sum_{\alpha\in R} {\mathfrak g}_{\alpha}  ,
    \end{equation*}
    Where ${\mathfrak k}^\C $ denotes the complexification of ${\mathfrak k}$, $R$ is the roots space and
    $$ {\mathfrak g}_{\alpha} := \left\{v \in {\mathfrak g}^\C \; \big| \; [H, v] = \alpha(H)v \; \forall\; H \in {\mathfrak k}\right\}.$$
    The algebra $\k^\C$ is also known as {\em Cartan subalgebra} and it is unique, up to conjugation by an automorphism of $\g$ \cite[Theorem 2.15]{MR1920389}.
    Moreover, the root spaces $\g_\alpha$ are one dimensional, and it holds \cite[Proposition 2.17]{MR1920389}
    \begin{equation}\label{eq: B orthogonal}
		\langle\g_\alpha,\g_\beta\rangle = 0 \quad \text{if } \alpha+\beta\neq0.
    \end{equation}
    The bracket relations between root spaces can be easily computed and are 
    \begin{equation}\label{eq: roots brackets relations}
		[\g_\alpha,\g_\beta] \begin{cases}
			=\g_{\alpha+\beta} & \text{if }\alpha+\beta \text{ is a non-zero-root} ;\\
			=0 & \text{if } \alpha+\beta  \text{ is not a root} ;\\
			\subset \k & \text{if } \alpha+\beta=0 .
		\end{cases} 
    \end{equation}
    Furthermore, if $\alpha$ is a root, then also $-\alpha$ is a root, and $\g_{-\alpha}=\bar{\g_\alpha}$.
    Therefore, one fix an element $H\in\k$ such that $\alpha(H)\neq0$ for every $\alpha\in R$ (which exists since $R$ is finite), and says that a root $\alpha$ is {\em positive} if $\im\alpha(H)>0$.
    A choice of the positive root system leads to the {\em Cartan decomposition}, namely $\g^\C$ decomposes as
    \begin{equation}\label{eq: cartan decomposition}
		{\mathfrak g}^\C = {\mathfrak k}^\C \oplus \sum_{\alpha\in R^+} {\mathfrak g}_{\alpha} \oplus {\mathfrak g}_{-\alpha} ,
    \end{equation}
    where $R^+$ is the space of positive roots.
	
\subsection{Samelson construction}\label{subsec: sam constr}
	
	
    The left-invariant almost-complex structures $J$ on $G$ are uniquely determined by their restriction to $\g$, which we still indicate with the same symbol.
    Hence we look at them as linear maps $J: {\mathfrak g}  \rightarrow {\mathfrak g}$ such that $J^2=-\id_\g$.
    Equivalently, an almost-complex structure on $G$ is determined by the subspace $\s \subset \g^\C$ of $(1,0)$-vectors, which clearly satisfies 
    $ \s \cap \g =0, \text{ and } \s \oplus \bar\s = \g^\C . $
    Finally, the integrability condition becomes $[\s,\s]\subset\s$, and thus, the complex structures on $G$ are in one-to-one correspondence with the complex Lie subalgebras $\s \subset \g^\C$, such that 
    $$[\s,\s]\subset\s\,,\quad \s \cap \g=0\,, \quad\text{and}\quad \s \oplus \bar\s = \g^\C .$$ 
    Such subspaces are called {\em Samelson subalgebras} of $\g^\C$ \cite{MR994129}. 
	
    Samelson \cite{MR59287} first constructed examples of left-invariant complex structures on compact Lie groups as follows.
    Consider the Cartan decomposition \eqref{eq: cartan decomposition}.
    Since $\dim(G)$ is even, the abelian Lie algebra ${\mathfrak k}$ is even-dimensional as well. Thus it is possible to choose a complex structure on $\k$. 
    As before, this is equivalent to choosing a complex subalgebra ${\mathfrak a} \subset {\mathfrak k}^\C$ such that 
    $$[ {\mathfrak a}, {\mathfrak a}]\subset\mathfrak{a}\,,\quad {\mathfrak a} \cap {\mathfrak k}=0\,,\quad\text{and}\quad {\mathfrak a} \oplus \overline{{\mathfrak a}} = {\mathfrak k}^\C .$$
    Now one could simply take
    \begin{equation*}
		{\mathfrak s} = {\mathfrak a} \oplus \sum_{\alpha\in R^+} {\mathfrak g}_{\alpha}
    \end{equation*}
    to be a Samelson subalgebra of $\g$.
    Thanks to the relation in \eqref{eq: B orthogonal}, the positive root spaces are orthogonal with respect to $\langle\cdot,\cdot\rangle$.
    Therefore, for a suitable choice of Samelson subalgebra on the torus the above complex structure is Hermitian together with the Killing metric. 
    
\subsection{Bi-invariant metrics compatible with Samelson complex structures}\label{subsec: bi-inv metrics}
	
    Let us suppose that $G=G_1\times\cdots\times G_s$ is a semisimple Lie group.
    The Lie algebra of the maximal torus splits as
    \[
		\k=\k_1\oplus\cdots\oplus\k_s
    \]
    with respect to the Killing metric.
    We say that a Samelson complex structure $J$ is {\em reducible} if there exist indexes $i_1,\ldots,i_p$ with $1\leq p<s$ such that 
    \[ J(\k_{i_1}\oplus\cdots\oplus\k_{i_p})\subset\k_{i_1}\oplus\cdots\oplus\k_{i_p}.\]
    It is {\em irreducible} if it is not reducible.
    Notice that a semisimple Lie group $G$ equipped with a Samelson complex structure $J$ can be always written as a product of Lie groups $G=G_1\times\cdots\times G_s$ such that the restriction of $J$ on any of the $G_i$ is an irreducible Samelson complex structure.

    Consider a bi-invariant metric $g$ on $G$ and the decomposition of the Lie algebra in simple pieces with respect to the Killing form
    $$\g = \g_1 \oplus \cdots \oplus \g_k.$$
    We claim that the $\g_i$ are $g$-orthogonal one to each other.
    As a matter of fact, consider the $1$-form on $G$ given by
    $$ q = g (x,\cdot) \quad \text{for an arbitrary } x\in\g_1. $$
    Then since $g$ is $\ad(G)$-invariant for any pair of elements $v,w\in\g_2\oplus\cdots\oplus\g_k$ it holds
    $$ q([v,w])=g(x,[v,w])=g([v,x],w) = 0. $$
    This implies that $q_{|_{\g_2\oplus\cdots\oplus\g_k}}=0$ since $[\g_2\oplus\cdots\oplus\g_k,\g_2\oplus\cdots\oplus\g_k]=\g_2\oplus\cdots\oplus\g_k$, and the same holds for the other $\g_i$'s.
    Moreover, the metric $g$ must be a multiple of the Killing form once restricted to any simple component $G_i$.
    Therefore, we can conclude that any bi-invariant metric $g$ on $G = G_1\times\cdots\times G_k$ is
    $$ g = - \lambda_1 B_1 - \lambda_2 B_2 - \cdots - \lambda_s B_k, $$
    for positive coefficients $\lambda_i\in\R_{>0}$ and $B_i$ the Killing form on $G_i$.
    
    To sum up, suppose that $G$ is a semisimple Lie group equipped with a bi-invariant metric $g$ and a compatible Samelson complex structure $J$ which decomposes in $s$ irreducible components.
    Thus $G = G_1\times\cdots\times G_s$ and $\left(G_i,J_{|_{G_i}}\right)$ are complex manifolds with irreducible Samelson complex structures $J_{|_{G_i}}$.
    Then the bi-invariant metrics on $G$ compatible with $J$ form an $s$-dimensional family since they are all of the form
    $$ \widetilde{g} = - \lambda_1 \,g_{|_{G_1}} - \lambda_2\, g_{|_{G_2}} - \cdots - \lambda_s\, g_{|_{G_s}}, $$
    for positive coefficients $\lambda_i\in\R_{>0}$.

\section{Cohomology of semisimple Lie groups}\label{sec: cohomology comp}

    Let $(G,J)$ be a semisimple Lie group of rank $r$ with maximal torus $T= \left(\S^1\right)^r \subset G$, and equipped with a Samelson complex structure $J$.
    The Tits fibration, defined as
    $$ \phi : G \rightarrow G/T: g \mapsto g\cdot T, $$
    gives a toric fibration
    $$ T \hookrightarrow G \xrightarrow{\phi} G/T$$
    over the {\em generalized flag manifold} $\left(G/T,J_{G/T}\right)$.
    We recall that the generalized flag manifolds are K\"ahler \cite{Borel}, and that they can be endowed with a (unique) invariant K\"ahler–Einstein metric \cite[Section 5]{MR303478}.
    Therefore, using Chern--Weil theory following \cite{MR1255937} it is possible to understand the Dolbeault cohomology of $(G,J)$.
    We thus combine it with the results from \cite{stelzig23} in order to compute the Aeppli cohomology of semisimple Lie groups.
    As we will see in the proof of Theorem \ref{th: global stability}, and in Remark \ref{rmk: crucial}, the Aeppli cohomology is particularly well-suited for studying pluriclosed metrics.
    As a matter of fact, the Aeppli cohomology groups of a complex manifold $(M,J)$ are defined as (see~\cite{MR0221536})
    $$ H^{p,q}_{A}(M,J)=\frac{\ker\left\{ \p\circ\bar\p:\A^{p,q}_J(M)\rightarrow\A^{p+1,q+1}_J(M)\right\}}{\imm\left\{\p:\A^{p-1,q}_J(M)\rightarrow\A^{p,q}_J(M)\right\}+\imm\left\{\bar\p:\A^{p,q-1}_J(M)\rightarrow\A^{p,q}_J(M)\right\}} . $$
    Therefore, the pluriclosed metrics always give a representative in the $H^{1,1}_A(M,J)$.\\

    Consider the connection one-form $\theta\in\A^1(G,\k)$ associated to the Tits fibration $\phi$.
    We indicate with $\omega\in\A^2(G/T,\k)$ the associated curvature form, $d\theta = \phi^*\omega$.
    We then define the maps on the dual of the Lie algebra of the torus
    $$\Phi : \k^* \rightarrow \A^1(G) : v \mapsto v \circ \theta$$
    and
    $$ D : \k^* \rightarrow  \A^2(G/T): v \mapsto  v\circ\omega,$$
    contracting respectively the toric components of $\theta$ and $\omega$.
    In particular, if we choose an orthonormal basis for $\k$ with respect to the Killing form 
    $$\k = \left<z_1,\ldots,z_r\right>$$
    the forms $\theta$ and $\omega$ can be written as
    $$\theta = (\theta_1,\ldots,\theta_r),\quad \text{and } \omega=(\omega_1,\ldots,\omega_r).$$
    Taking the dual basis 
    $$\k^* = \left<\xi^1,\ldots,\xi^r\right>,$$ 
    the maps $\Phi$ and $D$ are such that
    $$\Phi(\xi^i) = \theta_i, \quad\text{and}\quad D(\xi^i)=\omega_i.$$
    Now suppose we have a model for the Dolbeault and conjugate Dolbeault cohomologies of $G/T$
    $$ \Psi: \left(\mathcal{Y},\eth,\bar\eth\right) \xrightarrow{\sim} \left(\A^{\bullet,\bullet}(G/T),\p,\bar\p\right). $$
    The Dolbeault cohomology of $G/T$ satisfies $H^2(G/T) = H^{1,1}_{\bar\p}\left(G/T,J_{G/T}\right)$ \cite[Paragraph 14.10]{MR102800}.
    Therefore, by applying Proposition 8 of \cite{MR1255937} we obtain a model for the Dolbeault and conjugate Dolbeault cohomology of $G$ as follows:
    \begin{equation}\tag{\bf Tanré}\label{eq: Tanré model}
        \Psi\otimes\Phi : \left(\mathcal{Y}\otimes \bigwedge\left(\k^{1,0}\oplus\k^{0,1}\right), \delta, \bar\delta \right) \xrightarrow{\sim} \left(\A^{\bullet,\bullet}(G),\p,\bar\p\right),
    \end{equation}
    where $\left(\k^*\right)^\C = \k^{1,0}\oplus\k^{0,1}$ is the decomposition with respect to $J$, and $\delta,\bar\delta$ satisfy
    \begin{equation*}
        \delta v =
        \begin{cases}
            \eth v    & \text{if } v\in\mathcal{Y},\\
            0         & \text{if } v\in\k^{1,0},\\
            D(v)      & \text{if } v\in\k^{0,1};
        \end{cases}
        \quad
        \bar\delta v =
        \begin{cases}
            \bar\eth v    & \text{if } v\in\mathcal{Y},\\
            D(v)        & \text{if } v\in\k^{1,0},\\
            0           & \text{if } v\in\k^{0,1}.
        \end{cases}
    \end{equation*}
    As stated in \cite[Paragraph 14.10]{MR102800}, the Dolbeault cohomology of $G/T$ is all concentrated on bi-degrees $(p,p)$.
    It then immediately follows by applying the above model for the Dolbeault cohomology of $G$ that $H^{0,1}_{\bar\p}(G,J)=r$ and $H_{\bar\p}^{p,0}(G,J) = 0$ for all $p>0$.
    Notice that this argument was used in \cite[Lemma 6.1]{MR4032184} to prove that in particular the $(3,0)$-Dolbeault cohomology vanishes.
    
    We now clearify what is known about the Dolbeault cohomology of $G/T$, in order to take as $\Psi$ the simplest possible model.    
    We start by saying that the cohomology $H^*\left(G/T\right)$ is all generated by the unit and its $2$-dimensional classes, see \cite[Section 26]{MR51508}.
    These classes are precisely the ones of the $\omega_i$ as described in \cite[Proposition 14.6]{MR102800}, and the only relations among them are given by the polynomials $\R[\omega_1,\ldots,\omega_r]$ which are invariant with respect to the action of the {\em Weil group} $W_G:=N(T)/T$ (for $N(T)$ the normalizer of $T$ in $G$), see \cite[Proposition 27.1]{MR51508}.
    Finally, comparing the result in \cite[Proposition 27.2]{MR51508} with the de Rham cohomology of compact simple Lie groups one deduces that the first non-trivial relation among these $2$-forms occurs in degree $2$, in other words, it is a quadratic relation in the $\omega_i$.
    Since the Killing form is bi-invariant, and hence in particular invariant by the action of the Weil group, this quadratic relation must correspond to it in the following sense:
    $$\omega_1\wedge\omega_1 + \omega_2\wedge\omega_2 + \cdots + \omega_r\wedge\omega_r = 0 .$$
    In particular, for a semisimple Lie group, we have as many quadratic relations as the number of simple components: let us first distinguish the coordinates on the maximal torus as
    $$\k = \left<z_1^1,\ldots,z^1_{n_1},\ldots,z^s_1,\ldots,z^s_{n_s}\right>,$$
    where the upper indexes indicate the simple component they belong;
    if we arrange the characteristic classes accordingly, then the quadratic relations are all of the type
    \begin{equation}\label{eq: quadratic relations}
        \sum_{i=1}^s b_i\sum_{j=1}^{n_i}\omega^i_j\wedge\omega^i_j=0 \quad \text{for } b_i\in\R,
    \end{equation}
	i.\,e. corresponding to the bi-invarian metrics on $G=G_1\times\cdots\times G_s$.

    
    We now compute the $(1,1)$-Aeppli cohomology of $(G,J)$.
    We first restrict to the case of irreducible Samelson complex structure; then the general case will follow by a K\"unneth formula.

    \begin{theo}\label{th: 1,1 A cohomology}
        Let $G$ be a semisimple Lie group equipped with a bi-invariant metric $g$ and a compatible irreducible Samelson complex structure $J$.
        Then 
        $$H_{A}^{1,1}(G,J)\cong \C.$$
    \end{theo}
    \begin{proof}  
        First of all, fix $r$ the rank of $G$ and $s$ the number of simple components, $G=G_1\times\cdots\times G_s$.
        Then, notice that since the metric $g$ is bi-invariant it equals
        $$g = -\lambda_1 B_{|_{\g_1}} -\lambda_2 B_{|_{\g_2}} - \cdots -\lambda_s B_{|_{\g_s}} ,$$
        where $B$ is the Killing form of $G$ (see Section \ref{subsec: bi-inv metrics}).
        Therefore we can scale the coordinates $\left<z_j^i\right>$ 
        in order to obtain an orthonormal basis for $g$.
        Let us keep the same symbols for the coordinates $\left<z_i\right>$ their dual $\left<\xi_i\right>$ and the connection and curvature forms $\theta_i$ and $\omega_i$ for the sake of simplicity.

        Since \eqref{eq: Tanré model} is a model for Dolbeault and conjugate Dolbeault cohomology, thanks to the result of Stelzig \cite[Theorem C]{stelzig23}, it is also a model for Aeppli and {\em Bott--Chern} cohomologies.
        We thus can use this model to construct a bi-complex quasi-isomorphic to $\left(\A^{\bullet,\bullet}(G),\p,\bar\p\right)$ with respect to the Aeppli cohomology.
        Starting from the bi-complex freely-generated by the $\omega_i$'s and $\k^{1,0},\k^{0,1}$, namely
        $$ \left(\bigwedge\left(\omega_1,\ldots,\omega_r\right) \otimes \bigwedge\left(\k^{1,0}\oplus\k^{0,1}\right), \delta, \bar\delta \right) ,$$
        we obtain our bi-complex computing all cohomologies by erasing dots in the bi-degrees $(p,p)$ corresponding to the $W_G$-invariant polynomials in $\R[\omega_1,\ldots,\omega_r]$.
        This affects the lower degrees starting from bi-degree $(2,2)$ with the relations \eqref{eq: quadratic relations}.
        Since we are interested in the $(1,1)$-Aeppli cohomology we only care about the {\em squares} of the preliminary bi-complex which become L-shaped {\em zig-zags} after erasing dots in the bi-degrees $(2,2)$ as in the following figure (we refer to \cite[Section 1.1.4]{stelzig23} for the notation of squares and zig-zags).
        
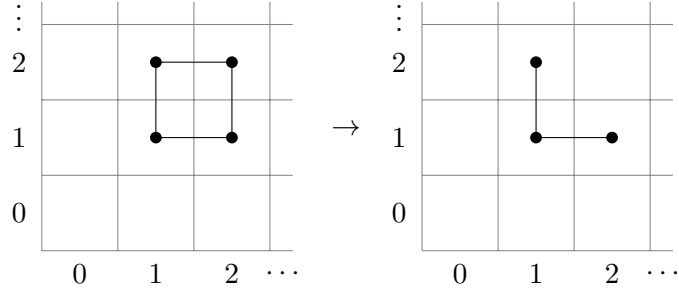
\begin{figure}[h!] \begin{center}\begin{tikzpicture}
            \newcommand\unito{1}
            
            \draw[help lines, step=\unito] (0,0) grid (3.3*\unito,3.3*\unito);
    
            \foreach \x in {0,...,2}
              \node at (\unito*0.5+\unito*\x,-0.3) {\x};
            \foreach \y in {0,...,2}
              \node at (-0.3,\unito*0.5+\unito*\y) {\y};
            \node at (-0.3, \unito*3.2) {$\vdots$};
            \node at (\unito*3.2, -0.3) {$\cdots$};

            \draw[help lines, step=\unito] (5*\unito,0) grid (8.3*\unito,3.3*\unito);

            \foreach \x in {0,...,2}
              \node at (\unito*5.5+\unito*\x,-0.3) {\x};
            \foreach \y in {0,...,2}
              \node at (\unito*5-0.3,\unito*0.5+\unito*\y) {\y};
            \node at (\unito*5-0.3, \unito*3.2) {$\vdots$};
            \node at (\unito*8.2, -0.3) {$\cdots$};

            \node at (\unito*4,\unito*1.6) {$\rightarrow$};

            \coordinate (11) at (3/2*\unito,3/2*\unito);
            \coordinate (12) at (3/2*\unito,5/2*\unito);
            \coordinate (21) at (5/2*\unito,3/2*\unito);
            \coordinate (22) at (5/2*\unito,5/2*\unito);

            \fill (11) circle (2.2pt);
            \fill (21) circle (2.2pt);
            \fill (12) circle (2.2pt);
            \fill (22) circle (2.2pt);
    
            \draw (11) -- (12);
            \draw (11) -- (21);
            \draw (21) -- (22);
            \draw (12) -- (22);

            \coordinate (11b) at (13/2*\unito,3/2*\unito);
            \coordinate (12b) at (13/2*\unito,5/2*\unito);
            \coordinate (21b) at (15/2*\unito,3/2*\unito);

            \fill (11b) circle (2.2pt);
            \fill (21b) circle (2.2pt);
            \fill (12b) circle (2.2pt);
    
            \draw (11b) -- (12b);
            \draw (11b) -- (21b);

            \end{tikzpicture} 
            \end{center}
        \end{figure} 
        
        We remark that all the forms $\omega_i$ are in $\imm\left(\delta\right)\cup\imm\left(\bar\delta\right)$ as
        \begin{equation}\label{eq: omega_i are aeppli-exact}
            2\omega_i = \delta\left(\xi^i + \im J \xi^i \right) + \bar\delta\left(\xi^i - \im J \xi^i \right).
        \end{equation}
        On the other hand, all the $(1,1)$-forms in $\k^{1,0}\wedge\,\k^{0,1}$ belong to a square in the preliminary bi-complex.
        More concretely, given an element
        $$ \sum_{i,j=1}^r A_i^j \left(\xi^i -\im J \xi^i \right) \wedge \left(\xi^j +\im J \xi^j \right) \in \k^{1,0}\wedge\k^{0,1},$$
        it is not $\delta$ nor $\bar\delta$ exact, and it holds
        \begin{multline*}
             -\bar\delta\delta\left( \sum_{i,j=1}^r A_i^j \left(\xi^i -\im J \xi^i \right) \wedge \left(\xi^j +\im J \xi^j \right) \right)
             = \bar\delta\left(\sum_{i,j=1}^r A_i^j \left(\xi^i -\im J \xi^i \right) \wedge D\left(\xi^j +\im J \xi^j \right) \right)\\
             = \bar\delta\left(\sum_{i,j=1}^r A_i^j \left(\xi^i -\im J \xi^i \right) \wedge \left(\omega^j +\im \sum_{q=1}^r J^j_q \omega^q  \right) \right)
             = \bar\delta\left( \sum_{i,j=1}^r \left(A_i^j +\im \sum_{q=1}^r J_j^q A_i^q \right) \left(\xi^i -\im J \xi^i \right) \wedge \omega^j  \right)\\
             = \sum_{i,j=1}^r \left(A_i^j +\im \sum_{q=1}^r J_j^q A_i^q \right) D\left(\xi^i -\im J \xi^i \right) \wedge \omega^j 
             = \sum_{i,j=1}^r \left(A_i^j +\im \sum_{q=1}^r J_j^q A_i^q \right) \left(\omega^i -\im \sum_{p=1}^r J^i_p \omega^p \right) \wedge \omega^j \\
             = \sum_{i,j=1}^r \left(A_i^j +\im \sum_{q=1}^r J_j^q A_i^q -\im \sum_{p=1}^r J^p_i A_p^j +\sum_{p,q=1}^r J^p_i J^q_j A_p^q    \right) \omega^i \wedge \omega^j \\
             = \sum_{i,j=1}^r \left(A -\im AJ -\im JA - JAJ \right)^j_i \omega^i \wedge \omega^j,
        \end{multline*}
        where $\left(J^j_k\right)$ is the matrix associated to the endomorphism $J$ in the coordinates $\left<z_i\right>$.
        We thus obtain a quadratic relation in the $\omega_i$'s.
        Hence we are now interested in understanding how many of these squares break in zig-zags giving representatives in $H^{1,1}_A(G,J)$.
        Thanks to equation \eqref{eq: quadratic relations}, this is equivalent to computing how many solutions there are of
        \begin{equation}\label{eq: crucial}
            Q + Q^t = 
            \begin{pmatrix}
            b_1 \id &         &        &         \\
                    & b_2 \id &        &         \\
                    &         & \ddots &         \\
                    &         &        & b_s \id
            \end{pmatrix}
        \end{equation}
        where $Q:=A-\im JA - \im AJ - JAJ$.
        Notice that $JQ = \im Q$ and $QJ = \im Q$.
        Therefore,
        \begin{equation}\label{eq: 1}
            -J\left( Q + Q^t\right)J = Q + Q^t .
        \end{equation}
        This forces the coefficients $b_i$ in \eqref{eq: crucial} to be all equal.
        Indeed, by multiplying \eqref{eq: 1} on the left by $J$ and using \eqref{eq: crucial} we have
        \[
            \begin{pmatrix}
            \cdots -b_1Jz_1^1 \cdots \\
            \vdots \\
            \cdots -b_kJz_k^l \cdots \\
            \vdots \\
            \cdots -b_sJz^s_{n_s} \cdots 
            \end{pmatrix} =
            \begin{pmatrix}
            \vdots & & \vdots & & \vdots \\
            b_1Jz_1^1 & \cdots & b_kJz^k_l & \cdots & b_sJz^s_{n_s} \\
            \vdots & & \vdots & & \vdots
            \end{pmatrix}.
        \]
        Therefore, if there exist two pairs $k=1,\ldots,s$, $l=1,\ldots, n_k$ and $p=1,\ldots,s$, $q=1,\ldots, n_p$ with $p\neq k$ such that $g(Jz_l^k,z^p_q)\neq 0$ it must be $b_k=b_p$.
        Finally, since the complex structure is irreducible we can guarantee this to happen and get $b_1=b_2=\cdots=b_s$.  

        A straightforward computation shows that there is no non-zero matrix $Q$ which is anti-symmetric and satisfies $JQ=\im Q=QJ$.
        Thus, there is only one solution (up to scalar multiplication) which is given by taking $A = \id$.
        It corresponds to the $(1,1)$-form in $\k^{1,0}\wedge\,\k^{0,1}$
        $$ 2\im\sum_{j=1}^r \xi^j\wedge J\xi^j ,$$
        which is in the same Aeppli cohomology class of the metric.
        Consequently, there is only one $(1,1)$-representative for the Aeppli cohomology of $(G,J)$, that is
        $$H^{1,1}_{A}(G,J) = \C\left<[ g(J\cdot,\cdot)]\right>.$$
    \end{proof}

    We know that $h^{0,1}_{\bar\p}(G,J)= r$ and $h^{p,0}_{\bar\p}(G,J)=0$ for $p\geq 1$, where $h^{\bullet,\bullet}_{\cdot\cdot}(G,J) = \dim H^{\bullet,\bullet}_{\cdot\cdot}(G,J)$.
    Therefore, in the bi-complex of $(G,J)$, there are $r$ zig-zags starting (horizontally) from bi-degree $(0,1)$. 
    These zig-zags can not have odd length (one or three) because that will count in de Rham cohomology which we know is $H^1(G)=0$.
    Moreover, they can not have length four because this would imply $H^{2,0}_{\bar\p}(G,J)\neq0$.
    Therefore they must have length equal to two, which implies 
    $$h_{\bar\p}^{1,1}(G,J) \geq h_{\bar\p}^{0,1}(G,J) = r.$$
    We recall that the $(1,1)$-forms coming from the torus belong to squares or L-shaped zig-zags, hence they do not contribute to Dolbeault cohomology in bi-degree $(1,1)$. 
    Consequently, the dimension of the $(1,1)$-Dolbeault cohomology equals the rank of $G$ since $H^{1,1}_{\bar\p}(G,J)$ is generated by the forms $\omega_i$ for $i=1,\ldots,r$ which satisfy \eqref{eq: omega_i are aeppli-exact}.

    We picture the information collected up to now in the diagram in Figure \ref{fig: doublecmplex G}, where we define $\nu_i := \xi^i +\im J \xi^i$ for $i=1,\ldots,r$.

    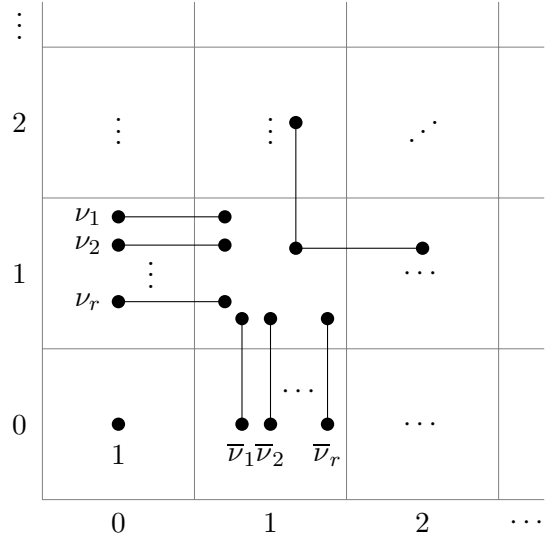
\begin{figure}[h!] \begin{center}\begin{tikzpicture}
        \newcommand\unito{2}
        
        \draw[help lines, step=\unito] (0,0) grid (3.3*\unito,3.3*\unito);

        \foreach \x in {0,...,2}
          \node at (\unito*0.5+\unito*\x,-0.3) {\x};
        \foreach \y in {0,...,2}
          \node at (-0.3,\unito*0.5+\unito*\y) {\y};

        \node at (-0.3, \unito*3.2) {$\vdots$};
        \node at (\unito*3.2, -0.3) {$\cdots$};
        \node at (2*\unito+1/2*\unito, 2*\unito+1/2*\unito) {$\iddots$};
        \node at (5/2*\unito,1/2*\unito) {$\cdots$};
        \node at (1/2*\unito,5/2*\unito) {$\vdots$};
        \node at (5/2*\unito,3/2*\unito) {$\cdots$};
        \node at (3/2*\unito,5/2*\unito) {$\vdots$};

        \coordinate (00) at (1/2*\unito,1/2*\unito);
        
        \coordinate (01) at (1/2*\unito,30/16*\unito);
        \coordinate (11b) at (6/5*\unito,30/16*\unito);
        \coordinate (01bis) at (1/2*\unito,27/16*\unito);
        \coordinate (11bbis) at (6/5*\unito,27/16*\unito);
        \coordinate (01tri) at (1/2*\unito,21/16*\unito);
        \coordinate (11btri) at (6/5*\unito,21/16*\unito);
        
        \coordinate (10) at (30/16*\unito,1/2*\unito);
        \coordinate (11a) at (30/16*\unito,6/5*\unito);
        \coordinate (10bis) at (24/16*\unito,1/2*\unito);
        \coordinate (11abis) at (24/16*\unito,6/5*\unito);
        \coordinate (10tri) at (21/16*\unito,1/2*\unito);
        \coordinate (11atri) at (21/16*\unito,6/5*\unito);
        
        \coordinate (11c) at (5/3*\unito,5/3*\unito);
        \coordinate (21) at (5/2*\unito,5/3*\unito);
        \coordinate (12) at (5/3*\unito,5/2*\unito);

        \fill (00) circle (2.5pt);
        \fill (01) circle (2.5pt);
        \fill (01bis) circle (2.5pt);
        \fill (01tri) circle (2.5pt);
        \fill (10) circle (2.5pt);
        \fill (10bis) circle (2.5pt);
        \fill (10tri) circle (2.5pt);
        \fill (11a) circle (2.5pt);
        \fill (11abis) circle (2.5pt);
        \fill (11atri) circle (2.5pt);
        \fill (11b) circle (2.5pt);
        \fill (11bbis) circle (2.5pt);
        \fill (11btri) circle (2.5pt);
        \fill (11c) circle (2.5pt);
        \fill (12) circle (2.5pt);
        \fill (21) circle (2.5pt);
        
        \node at (1/2*\unito,1/2*\unito-.4) {$1$};
        \node at (-.4+1/2*\unito,30/16*\unito) {$\nu_1$};
        \node at (-.4+1/2*\unito,27/16*\unito) {$\nu_2$};
        \node at (-.4+1/2*\unito,21/16*\unito) {$\nu_r$};
        \node at (5/7*\unito,.1+3/2*\unito) {$\vdots$};
        \node at (30/16*\unito,1/2*\unito-.4) {$\bar\nu_r$};
        \node at (24/16*\unito,1/2*\unito-.4) {$\bar\nu_2$};
        \node at (21/16*\unito,1/2*\unito-.4) {$\bar\nu_1$};
        \node at (3/2*\unito+.4,5/7*\unito) {$\cdots$};

        \draw (10) -- (11a);
        \draw (01) -- (11b);
        \draw (01bis) -- (11bbis);
        \draw (01tri) -- (11btri);
        \draw (10bis) -- (11abis);
        \draw (10tri) -- (11atri);
        \draw (21) -- (11c) -- (12);

        \end{tikzpicture} 
            \caption{Zig-zags of the double complex of $(G,J)$ involving bi-degrees $(1,1)$ or lower, up to squares.}
            \label{fig: doublecmplex G}
        \end{center}
    \end{figure}

    In order to extend Theorem \ref{th: 1,1 A cohomology} to a semisimple Lie group equipped with a general Samelson complex structure we use the K\"unneth formula for the Aeppli cohomology computed in \cite{stelzig23}.
    Notice that the following result can also be obtained with a direct computation as in Theorem \ref{th: 1,1 A cohomology} without using K\"unneth's formula.
    
    \begin{cor}\label{cor: 1,1 A cohomology reducible case}
        Let $G$ be a semisimple Lie group equipped with a bi-invariant metric $g$ and a compatible Samelson complex structure $J$.
        Suppose that $J$ decomposes in $s$ irreducible components, then 
        $$H_{A}^{1,1}(G,J)\cong \C^s.$$
    \end{cor}
    \begin{proof}
        The complex manifold $(G,J)$ decomposes as a product $G=G_1\times\cdots\times G_s$ of complex manifolds $\left(G_i,J_{|_{G_i}}\right)$ where $J_{|_{G_i}}$ is an irreduble Samelson complex structure.
        Then the result follows from Theorem \ref{th: 1,1 A cohomology} using the K\"unneth formula for Aeppli cohomology \cite[Corollary 1.36]{stelzig23}.
        Indeed, using the same notations as in \cite{stelzig23}, we have the following short exact sequence
        $$ 0 \rightarrow L \rightarrow H_A\left(G_1\times\cdots\times G_s,J\right) \rightarrow \left(\ker\p\bar\p\right)\left(H_A\left(G_1,J_{|_{G_1}}\right)\otimes \cdots \otimes H_A\left(G_s,J_{|_{G_s}}\right)\right) \rightarrow 0 .$$
        In bi-degree $(1,1)$, the bi-complex structure (partially pictured in Figure \ref{fig: doublecmplex G}) is such that $L$ vanishes and hence
        \begin{multline*}
            H^{1,1}_A\left(G_1\times\cdots\times G_s,J\right) \cong \left[H_A^{1,1}\left(G_1,J_{|_{G_1}}\right)\otimes H_A^{0,0}\left(G_2,J_{|_{G_2}}\right)\otimes \cdots \otimes H_A^{0,0}\left(G_s,J_{|_{G_s}}\right)\right] \oplus \cdots \\
            \oplus \left[H_A^{0,0}\left(G_1,J_{|_{G_1}}\right)\otimes \cdots \otimes H_A^{0,0}\left(G_{s-1},J_{|_{G_{s-1}}}\right)\otimes  H_A^{1,1}\left(G_s,J_{|_{G_s}}\right)\right]
        \end{multline*}

    \end{proof}

    
    \begin{rmk}\label{rmk: crucial}
        Since the only non-zero representatives in $H^{1,1}_A(G,J)$ comes from {\rm L}-shaped zig-zags, by counting the bi-invariant metrics, it can be proved that $s\leq h^{1,1}_A(G,J)\leq h^{2,1}_{\bar\p}(G,J)$.
        Therefore, the results of Theorem \ref{th: 1,1 A cohomology} and Corollary \ref{cor: 1,1 A cohomology reducible case} would follow by proving that $h^{2,1}_{\bar\p}(G,J)=s$. 
        However, unless some particular cases (for example for rank less or equal than $4$ or some specific complex structures) we are not able to prove it. 
        Somehow, for these manifolds, it is, in general, easier to explicitly compute the Aeppli cohomology than the Dolbeault cohomology.
        More precisely, the structure of the bi-complex computing all the cohomologies remains a black box from which we are only able to recover information about the {\rm L}-shaped zig-zags coming from squares.
        We want to highlight that this is unusual since we have much more tools to compute the de Rham and Dolbeault cohomologies, while one usually recovers information on the Aeppli and Bott--Chern cohomologies from the formers (as for example in \cite{MR3340173,barbaro_pcflow}). 
        Also in \cite{MR4032184} the authors use the Tanré's model to get information on the $(2,1)$ and $(3,0)$ Dolbeault cohomology groups and from that deduce the (non-)existence of pluriclosed metrics.
        However, the $(1,1)$-Aeppli cohomology is clearly better suited for studying pluriclosed metrics. It is therefore interesting that here we use Tanré model to directly compute the Aepply cohomology and by-pass the understanding of the Dolbeault cohomology.
    \end{rmk}
 
\section{Global stability}\label{sec: stability}

    Bismut flat metrics are attractive (in the sense of the following theorem) for the pluriclosed flow in their {\em torsion classes}, which, for a metric $\omega$, is the class of $\p\omega$ in the $(2,1)$-Dolbeault cohomology.
    \begin{theo}[Theorem 1.2 of \cite{Garcia-Fernandez:2021tq}]\label{thm: GJS convergence}
        Let $(M,J, \omega_{BF})$ be a compact Bismut flat manifold. Given $\omega_0$ a pluriclosed metric such that $[\p\omega_0] = [\p\omega_{BF} ] \in H^{2,1}_{\bar\p}(M,J)$, the solution of the pluriclosed flow with initial data $\omega_0$ exists on $[0, \infty)$ and converges to a Bismut flat metric $\omega_\infty$.
    \end{theo}


	Thanks to the knowledge of the $(1,1)$-Aeppli cohomology of compact Samelson spaces achieved in the previous section, we can now check that the cohomological condition in Theorem \ref{thm: GJS convergence} holds for all the pluriclosed metrics on these manifolds.
	It follows that the Bismut flat metrics on compact Bismut flat manifolds with finite fundamental group are globally stable for the pluriclosed flow. 
    We use here a refined version of the argument from \cite{barbaro_pcflow} which applies to this much more general case, proving the following result.
    \begin{theo}\label{th: global stability}
        Let $(G,J,\omega_{BF})$ be a compact Samelson space, that is a compact simply-connected semisimple Lie group with a Bismut flat Hermitian structure coming from the Killing form (as described in Section \ref{sec: 2}). 
        Suppose it decomposes in $s$ irreducible components $\left(G_i,J_{|_{G_i}}\right)$ and define $\omega_i:=\left(\omega_{BF}\right)_{|_{G_i}}$.
        Then for any pluriclosed metric $\omega_0$ on $\left(G,J\right)$ there exist positive constants $\gamma_i$'s such that the solution to the pluriclosed flow with initial data $\omega_0$ exists on $[0, \infty)$ and converges to $\sum_{i=1}^s\gamma_i\omega_i$ up to diffeomorphism. 
    \end{theo}
    \begin{proof}
        First of all, given a pluriclosed metric $\omega$ on $\left(G,J\right)$ we show that its class in the $(1,1)$-Aeppli cohomology can not be zero, see also \cite[Theorem 1.1]{marouani2023}.
        Suppose that $\omega$ is $\left(\p \,\rm{or}\, \bar\p\right)$-exact, namely there exists $\alpha\in\A^{1,0}_JG$ such that $\omega = \p\bar\alpha + \bar\p\alpha$.
        Then 
        $$\Omega:=\omega+\p\alpha+\bar\p\bar\alpha = d(\alpha+\bar\alpha)$$
        is a symplectic form since it is obviously a closed $2$-form, and for any non-zero vector field $x\in\Cinf(G;TG)$ it holds
        $$ \Omega(x,Jx) = \omega(x,Jx)+\p\alpha(x,Jx)+\bar\p\bar\alpha(x,Jx) = g(x,x) \neq 0. $$
        However, a result of Borel \cite[Theorem 1]{Borel} ensures that $G$ does not admit any symplectic structure.\\
		
		Now, following Section \ref{subsec: bi-inv metrics}, the restriction of $\omega_{BF}$ to any $\left(G_i,J_{|_{G_i}}\right)$ is a Bismut flat metric $\omega_i=\left(\omega_{BF}\right)_{|_{G_i}}$, which thanks to the above considerations is a representative of a non-zero class in $H^{1,1}_A\left(G_i,J_{|_{G_i}}\right)$.
        Moreover, Corollary \ref{cor: 1,1 A cohomology reducible case} ensures that $H^{1,1}_A(G,J)$ is generated by the pull-backs of these metrics.
        Thus for any pluriclosed metric $\omega_0$ on $(G,J)$ there exist constants $\lambda_i$ for $i=1,\ldots,s$ such that 
        $$[\omega_0] = \sum_{i=1}^s \lambda_i [\omega_{i}] \quad\text{in } H^{1,1}_A \left(G,J\right).$$ 
        Here the constants $\lambda_i$'s must be positive because the calss of $\left(\omega_0\right)_{|_{G_i}} - \lambda_i\omega_i$ is zero in $H^{1,1}_A\left(G_i,J_{|_{G_i}}\right)$ hence it can not be a pluriclosed metric on $\left(G_i,J_{|_{G_i}}\right)$.
        Consequently, by applying the $\p$ operator one gets that
        $$[\p\omega_0] = \left[ \p \sum_{i=1}^s \lambda_i \omega_{i}\right] \quad\text{in } H^{2,1}_{\bar\p} \left(G,J\right).$$
        Finally, we notice that the metric $\sum_{i=1}^s \lambda_i \omega_{i}$ is a Bismut flat metric on $(G,J)$ since it is bi-invariant and compatible with $J$.
        Therefore, Theorem \ref{thm: GJS convergence} applies to ensure the long-time existence of the pluriclosed flow with initial data $\omega_0$ and convergence to a Bismut flat metric $\omega_\infty \in \sum_{i=1}^s \lambda_i [\omega_{i}]_A$.\\
        
        Thanks to the characterization in Section \ref{sec: 2}, Bismut flat metrics are bi-invariant with respect to some Lie group structure, and all the bi-invariant metrics with respect to the action of $G$ which are compatible with $J$ are of the type $\sum_{i=1}^s \gamma_i \omega_{i}$ for positive constants $\gamma_i$'s.
        We notice that a priori there might be different Lie group structures on $G$ as manifold.
        However, the Milnor result (Lemma \ref{lem: Milnor}) ensures that there is only one up to diffeomorphisms of the underlying manifold.
        Consequently,  $\omega_\infty$ and $\omega_{BF}$ are bi-invariant with respect to two isomorphic Lie group structures on $G$. 
        In particular, there exists $\varphi\in\mbox{Diff}(G)$ such that $\varphi^*\omega_\infty=\sum_{i=1}^s \gamma_i \omega_{i}$ for some positive constants $\gamma_i$'s.
    \end{proof}

    As an immediate corollary of Theorem \ref{th: global stability} we obtain the global stability of the Bismut flat metrics for the pluriclosed flow on any manifold whose universal cover is a compact Samelson space.
    Indeed, we can prove the global existence of the pluriclosed flow and its convergence on the universal cover.
    Thanks to the classification in Theorem \ref{thm: Bismut flat manifolds} these manifolds are precisely the compact Bismut flat manifolds with finite fundamental group.

    \begin{theo}\label{th: global stability bis}
        Let $(M,J,\omega_{BF})$ be a compact Bismut flat manifold with finite fundamental group.
        Then for any pluriclosed metric $\omega_0$ on $\left(M,J\right)$ the solution to the pluriclosed flow with initial data $\omega_0$ exists on $[0, \infty)$ and converges to a Bismut flat metric $\omega_{\infty}$.
    \end{theo}

	\medskip

    The case of compact Bismut flat manifolds with infinite fundamental group (which by Theorem \ref{thm: Bismut flat manifolds} corresponds to manifolds diffeomorphic to $G\times\left(\S^1\right)^k$ for a semisimple Lie group $G$ and positive $k$) remains open.
    In fact, the main argument we used in proving global stability was to show that the subspace generated by the classes of the Bismut flat metrics fills up the whole $H^{1,1}_A(G,J)$.
    However, for {\em local} Samelson space the subspace of $H^{1,1}_A(G,J)$ generated by the classes of the pluriclosed metrics grows in such a way that it makes more difficult to verify that it is generated by Bismut flat metrics.
    We present here an example in order to clarify this difficulty.
    
    \begin{ex}\label{ex: higher rank}
        Consider the simple Lie group of rank two $\SU(3)$ equipped with the Bismut flat Hermitian structure $\left(J_{\SU(3)},\omega_{BF}\right)$ coming from the Killing form, and the complex torus with the standard complex structure
        $\left(\S^1\times\S^1,J_{St}\right)$.
        Define $J:=J_{\SU(3)}\times J_{St}$ the product complex structure on $M:=\SU(3)\times\S^1\times\S^1$.
        The lower bi-degrees of the bi-complex associated to $\left(\SU(3),J_{\SU(3)}\right)$ are as in Figure \ref{fig: doublecmplex G}.
        Therefore, by applying the K\"unneth formula for Aeppli cohomology \cite[Corollary 1.36]{stelzig23} one gets that the $(1,1)$-Aeppli cohomology of $(M,J)$ is of dimension $4$ generated by the classes of $\omega_{BF}, \psi\wedge\bar\psi, \nu_1\wedge\bar\psi, \psi\wedge\bar\nu_1$, for $\psi$ the $(1,0)$-form generating the cohomology of the torus.
        Thus, given a generic pluriclosed metric $\omega$ on $(M,J)$, its class in the $(1,1)$ Aeppli cohomology group $H^{1,1}_A(M,J)$ is
        \begin{equation}\label{eq: above}
            [\omega] =  \alpha[\omega_{BF}] + \beta \frac{\im}{2}[\psi\wedge\bar\psi] + u [\nu_1\wedge\bar\psi] - \bar u [\psi\wedge\bar\nu_1],
        \end{equation}
        for coefficients $\alpha,\beta\in\R$, and $u\in\C$ such that
        \begin{equation}\label{eq: bove}
        \begin{cases}
            \alpha>0, \quad \beta>0,\\
            \alpha\beta>4|u|^2.
        \end{cases}\end{equation}
        The above equation \eqref{eq: above} defines the subset of $H^{1,1}_A(M,J)$ generated by the classes of the pluriclosed metrics on $(M,J)$.
        Henceforth, to apply the argument in Theorem \ref{th: global stability} one should verify that it is generated by the classes of the Bismut flat metrics, however, we can not argue as before counting the bi-invariant metrics coming from the Killing forms of the simple components.
        We thus define a family of pluriclosed metrics as
        \begin{align*}
            \omega_{\alpha,\beta,u} := \alpha\omega_{BF} + \beta \frac{\im}{2}\psi\wedge\bar\psi + u \nu_1\wedge\bar\psi - \bar u \psi\wedge\bar\nu_1, 
        \end{align*}
        with coefficients $\alpha,\beta$ and $u$ as in \eqref{eq: bove}, and by straightforward computations, we verify that all these metrics are Bismut flat.

        Note that the metrics $\omega_{\alpha,\beta,u}$ give Bismut flat metrics on the cover $\SU(3)\times\R^2$ which are not bi-invariant with respect to the product Lie group structure whenever $u\neq0$.
        Indeed, being bi-invariant means that for any $x,y,z$ in the Lie algebra
        $$\omega_{\alpha,\beta,u} ([x,y],Jz) + \omega_{\alpha,\beta,u} ([x,z],Jy) = 0.$$
        However, taking $z=\psi^*$ the dual of $\psi$ we get
        \begin{equation*}
            \omega_{\alpha,\beta,u} ([\cdot,\cdot\cdot],J\psi^*) + \omega_{\alpha,\beta,u} ([\cdot,\psi^*],J\cdot\cdot) = 
            \im \bar u \,\bar\nu_1([\cdot,\cdot\cdot]) = - \im \bar u \,d\,\bar\nu_1 \neq 0.
        \end{equation*}
        A posteriori, thanks to the characterization of Bismut flat metrics of Theorem \ref{thm: Bismut flat manifolds}, we may assert that they are bi-invariant with respect to another Lie group structure on $\SU(3)\times\R^2$ (as manifold) isomorphic to the previous one.
        More precisely, by Lemma \ref{lem: Milnor}, for any $u\in\C$ there exists an isometry $$\phi_u : \left(\SU(3)\times\R^2,\omega_{\alpha,\beta,u}\right) \xrightarrow{\sim} \left(\SU(3)\times\R^2,\omega_{\alpha',\beta',0}\right).  $$
    \end{ex}

\section{Homogeneous Bismut Hermitian Einstein metrics}\label{sec: non-flat BHE}

    The existence of non-flat Bismut Hermitian Einstein metrics on Bismut flat manifolds with finite fundamental group is obstructed by the global stability property for the pluriclosed flow proved in Theorem \ref{th: global stability bis}.
    Indeed, we have the following simple corollary of Theorem \ref{th: global stability bis}.
    \begin{cor}\label{cor: obstruction on Bf}
        Let $(M,J,\omega_{BF})$ be a compact Bismut flat manifold with finite fundamental group.
        Then any Bismut Hermitian Einstein metric is Bismut flat.
    \end{cor}
	\begin{proof}
		Suppose $\omega_{BHE}$ is a Bismut Hermitian Einstein metric on $(M,J)$.
		By Theorem \ref{th: global stability bis} it must converge through the pluriclosed flow to a Bismut flat metric. However, $\omega_{BHE}$ is a static point of the flow, hence it must be Bismut flat itself.
	\end{proof}

    Now let $(M,J)$ be a C-space, which we recall is a compact complex manifold admitting a transitive action by a compact Lie group of biholomorphisms and finite fundamental group.
    Thanks to \cite[Theorem C]{MR66011}, such a space admits a transitive action of a compact semisimple Lie group, hence it can be written as $M=G/H$ where $G$ is a compact semisimple Lie group and $H$ a closed connected subgroup whose semisimple part coincides with the semisimple part of the centraliser of a toral subgroup of $G$. 
    Consequently, $G/H$ fibers over a generalized flag manifold $G/T$ with toric fibers:
    $$T/H \hookrightarrow G/H \rightarrow G/T.$$
    The existence of a pluriclosed metric and the vanishing of the Chern class together give further restriction on the geometry of these manifolds which we describe in the following result.
     \begin{prop}\label{prop: c-space}
    	Let $(M,J)$ be a C-space, and suppose that it can be equipped with a Bismut Hermitian Einstein metric.
    	Then, up to finite cover, it is a compact semisimple Lie group equipped with a Samelson complex structure. 
    \end{prop}
	\begin{proof}
		First of all, Theorem 6.1 in \cite{MR4032184} ensures that every pluriclosed C-space is (up to a finite cover) the product of a compact Lie group and a generalized flag manifold.
		Then, the existence of a CYT metric on $(M,J)$ would force the first Chern class to vanish in the $(1,1)$-Aeppli cohomology.
		In fact, for a Hermitian manifold $(M,J,g)$ it holds \cite[(2.7)]{MR1836272}
		\begin{equation}\label{eq: ricci bismut e chern}
			\left(Ric^B(\omega)\right)^{1,1} = Ric^{Ch}(\omega) + \im\left(\p\theta^{0,1} - \bar\p\theta^{1,0} \right),
		\end{equation}
		where $\theta^{1,0}$ and $\theta^{0,1}$ are respectively the $(1,0)$ and $(0,1)$ components of the Lee form. 
		However, generalized flag manifolds are K\"ahler--Einstein Fano \cite[Section 5]{MR303478}, while it is known that the first Chern class of Lie groups equipped with Samelson complex structures vanishes in de Rham (hence also Aeppli) cohomology, see e.\,g. \cite[Theorem 2]{MR2764884}.
		Therefore, there can not be any flag component in the cover.
	\end{proof}


    We remark that not any Samelson complex structure is compatible with the existence of a Bismut flat metric; indeed, only the ones which are Hermitian with respect to a bi-invariant metric are so. 
    However, the existence of a left-invariant pluriclosed metric on a compact semisimple Lie group endowed with a Samelson complex structure forces the manifold to admit a Bismut flat metric \cite[Theorem 3.2]{Fino22}.    
    
    To summarize, a C-space $(M,J)$ can be equipped with a left-invariant Bismut Hermitian Einstein metric (if and) only if it is a Bismut flat manifold.
    The `if' implication trivially follows by the fact that Bismut flat metric are bi-invariant and Bismut Hermitian Einstein.
    We thus reduced to the setting where we can apply Corollary \ref{cor: obstruction on Bf} to get the following result.
    \begin{theo}\label{thm: non-existence on C-spaces}
        Let $(M, J)$ be a C-space and suppose that $g$ is a left-invariant Bismut Hermitian Einstein metric on it. Then, up to finite cover, $M$ is a Lie group and $g$ is a bi-invariant metric. In particular, $g$ is Bismut flat.
    \end{theo}
	\begin{proof}
		Thanks to Proposition \ref{prop: c-space}, $M$ is finitely covered by a compact semisimple Lie group and $J$ lifts to a Samelson complex structure.
		Then, by \cite[Theorem 3.2]{Fino22} $J$ must be Hermitian with respect to a bi-invariant metric.
		Consequently, $(M,J)$ admits a Bismut flat metric, and we can apply Corollary \ref{cor: obstruction on Bf} to conclude.
	\end{proof}

\printbibliography

\end{document}

%% file: configuration.tex
\usepackage[utf8]{inputenc}
\usepackage{csquotes}
    \PassOptionsToPackage{%
    	backend=bibtex8,bibencoding=ascii,%
    	language=auto,%
    	style=numeric-comp,%
        sorting=nyt, 
        maxbibnames=10, 
        natbib=true 
    }{biblatex}
    \usepackage[backend=biber]{biblatex}
\usepackage{amssymb,mathtools,amsfonts,latexsym,rawfonts, mathrsfs, lscape}
\usepackage{stmaryrd,tikz,pgfplots}
\usepackage{marginnote}  
\usepackage{hyperref}
\usepackage[all]{xy}
\usepackage{amsthm,latexsym,amssymb}
\usepackage{tensor}

\usepackage{multirow}

\usepackage{array, tabularx}

\usepackage{booktabs} 
\usepackage{siunitx}

\usepackage{graphicx}
\usepackage{subfigure}

\usepackage{tikz-cd}
\usepackage{tikz}
\usepackage{color}

\usepackage{dynkin-diagrams}
\usepackage{mathdots}
\usepackage{tabularx}

\usepackage[top=1.5in, bottom=.9in, left=0.9in, right=0.9in]{geometry}

\newtheorem{theo}{Theorem}[section]
\newtheorem*{theo*}{Theorem}
\newtheorem{defi}{Definition}[section]
\newtheorem*{defi*}{Definition}
\newtheorem*{notation*}{Notation}
\newtheorem{lem}{Lemma}[section]
\newtheorem*{lem*}{Lemma}
\newtheorem*{obs*}{Observation}
\newtheorem{cor}{Corollary}[section]
\newtheorem*{cor*}{Corollary}
\newtheorem{prop}{Proposition}[section]
\newtheorem*{prop*}{Proposition}

\newtheorem*{conj*}{Conjecture}

\theoremstyle{definition}
\newtheorem{ex}{Example}[section]
\newtheorem*{ex*}{Example}
\newtheorem{rmk}{Remark}[section]
\newtheorem*{rmk*}{Remark}

\renewcommand{\[}{\begin{equation*}}
\renewcommand{\]}{\end{equation*}}

\def\C{\mathbb{C}}

\def\Cinf{\mathcal{C}^\infty}

\DeclareMathOperator\tr{tr}
\def \ad {\mbox{ad}}
\DeclareMathOperator{\im}{i}
\DeclareMathOperator{\imm}{Im}

\DeclareMathOperator\id{id}

\def \A {\mathcal A}

\def \R {\mathbb R}
\def \S {\mathbb S}
\def \C {\mathbb C}
\def \Z {\mathbb Z}

\def \SU {\mathrm{SU}}

\def \k {\mathfrak k}
\def \g {\mathfrak g}

\def \s {\mathfrak s}

\def \Na {\nabla}

\newcommand{\p}{\partial}
\renewcommand{\bar}{\overline}